\documentclass[12pt,a4paper,leqno]{amsart}

\usepackage[utf8]{inputenc}
\usepackage{amsfonts, amsmath, amssymb, amsthm}
\usepackage{indentfirst}
\usepackage{xcolor}

\usepackage[T1]{fontenc}
\renewcommand{\epsilon}{\varepsilon}

\newcommand{\calG}{\mathcal{G}}

\newcommand{\cl}{\operatorname{cl}}

\newcommand{\id}{\operatorname{id}}

\renewcommand{\epsilon}{\varepsilon}
\renewcommand{\int}{\operatorname{int}}
\renewcommand{\phi}{\varphi}

\newtheorem{thm}{Theorem}
\newtheorem{pro}[thm]{Proposition}
\newtheorem{lem}[thm]{Lemma}
\newtheorem{cor}[thm]{Corollary}

\theoremstyle{definition}

\title{Homeomorphisms between compact subsets of   real numbers}

  
\author{Sławomir Kusiński}
\address{Faculty of Applied Mathematics,
Silesian University of Technology, Kaszubska 23 , 44-100 Gliwice, Poland}

\email{Slawomir.Kusinski@polsl.pl}

\author{Szymon Plewik}
\address{Institute of Mathematics, University of Silesia in Katowice, Bankowa~14, 40-007 Katowice, Poland}

\email{szymon.plewik@us.edu.pl}

\begin{document}

  \begin{abstract} We first present  a reduction of properties  of  compact sets of real  numbers to properties of countable orders. Then discuss  a variant of homogeneity  of compact subsets of  real numbers, focusing on the family of  $t$-sets. Finally, we prove that there are exactly $\omega_1$ many non-homeomorphic $t$-sets. \end{abstract}

  \maketitle
  \section{Introduction}
	The preliminary purpose of this paper is to collect  reductions of properties of  compact subsets of real  numbers to properties of countable orders. In this respect, the  paper is a continuation of the research discussed in  \cite{bkp}. Relevant invariants are applied to the so-called $t$-sets in $\mathbb R$, which allows for their full classification up to homeomorphism. As the readers will note, this classification is demonstrated by the method of exhausting the possibilities through the order in which the examples are discussed.

 If $Y$ is a  closed subset of real  numbers $\mathbb{R}$, then any   component of its complement  is called  a $Y$-\textit{gap} and any non-trivial component of $Y$ is  called a $Y$-\textit{interval}. An endpoint $e$ of a $Y$-interval $I$ is called \textit{free} whenever 
$e \in \int_Y I$: in other words $e\notin \cl_Y(Y\setminus I)$. Let $\calG_Y$, $\mathcal B_Y$, $\mathcal C_Y$ and $\mathcal D_Y$ be the families of all $Y$-gaps,  all $Y$-intervals with no free endpoint, all $Y$-intervals with  exactly one free endpoint and all $Y$-intervals with two free endpoints, respectively. 
The family $$\mathcal Q_Y= \mathcal G_Y\cup\mathcal B_Y\cup \mathcal C_Y\cup \mathcal D_Y $$ is countable and consists of pairwise disjoint intervals. We will use the abbreviation $\mathcal I_Y=\mathcal B_Y\cup \mathcal C_Y\cup \mathcal D_Y $. If $I$ and $J$ are distinct intervals, then   we  write $I < J$ if and only if $\sup I \leqslant \inf J$. It won't be misleading when $<$  
  denotes the strict inequality between real numbers, too. 
	
Any linearly ordered space $(\mathcal Q_Y, <)$ has a few obvious properties. For example, when $Y$ is compact, then the smallest and largest elements of $\mathcal Q_Y$ exist and belong to $\mathcal G_Y$. Also, between any two different  $Y$-intervals there is  a $Y$-gap.

If $(\mathcal Q,<)$ is an ordered space, then we say that $\mathcal G \subseteq \mathcal Q$ \textit{interlaces} $\mathcal B\subseteq \mathcal Q$ whenever for every two distinct elements $I<J$ of $\mathcal B$ there is $K\in\mathcal G$ with $I<K<J$.  In the special case $\mathcal B=\mathcal Q$, this is just density of $\mathcal G$ in  $\mathcal Q$.
For a compact set $Y\subset \mathbb R$, the family of all $Y$-gaps interlaces the family of all $Y$-intervals.

A set $Y\subseteq \mathbb R$ is regularly closed if and only if  $Y$ is closed and the union of all $Y$-intervals is dense in $Y.$ 
As noted in \cite{bkp}, some  topological properties of a regularly closed subset $Y\subseteq \mathbb R$ can be described by  properties of the ordered space $(\mathcal Q_Y, <)$. This includes results from \cite{bkp} such as Theorem 9, Proposition 10, Corollaries 12, 13 and 15,  and Lemma  16. 

 If $\mathcal X$ is a set which is  linearly ordered by a  relation $<$, then $(\mathcal X,<)$ is briefly called an \textit{ordered space}.  The topology induced on $\mathcal X$ by $<$  is called the \textit{order topology}, see \cite[pp. 56-57]{enge}  and compare \cite[p. 66]{ss}.  An increasing bijection between ordered spaces is called an \textit{isomorphism}. If  $(\mathcal Q, <)$ and $(\mathcal P, <)$ are ordered spaces, then the \textit{lexicographic order} on  $\mathcal Q \times \mathcal P$  is defined by $(a,b) < (c,d)$ whenever $a<c$, or  whenever $a=c$ and $b<d$.  The remaining notions relevant to our purposes are defined just before we use them for the first time. Other  notions and notation are standard, compare \cite{bkp}, \cite{enge}, \cite{si} or \cite{ss}. 
  
	As W. Kulpa told us, it is worth noting that the now obvious fact that: \textit{The continuous bijection of the real numbers onto the real numbers is
	monotone}; should be attributed to B. Bolzano \cite{bo}, as an obvious consequence of the Bolzano intermediate value theorem. We use this fact, without recalling it, in the version that for any bijection $f\colon \mathbb R \to \mathbb R$, being  monotone  is equivalent to being a homeomorphism. 

A few  books and articles  discuss so-called tame sets to explain the phenomenon  of Antoine's necklace.   Fascinated  by Bing's claim: \textit{each
Cantor set in a line or a plane is tame}; see \cite[p. 435]{bi}, we restrict ourselves to sets that are tame in $\mathbb R$. As J. Mioduszewski noted 
\cite[p. 161]{mi},  the French version of K. Kuratowski's monograph \cite[pp. 341--385]{kur} discusses the fact that any compact subset of $\mathbb R$ is tame in $\mathbb R^2$, what has been first proved in  \cite{mk}. We have observed that   tame sets in $\mathbb R$, hereinafter we will call them $t$-sets, are discussed  in  so-called mathematical folklore, i.e., we did not find their comprehensive description in the literature.
   \section{Reverse reduction} 
  Suppose that a countable  ordered space $(\mathcal Q, <)$ is given such that $\mathcal Q = \mathcal G \cup \mathcal I$ and 
$\mathcal G \cap \mathcal I =\emptyset$. We analyze what assumptions need to be made to be able to find a compact subset $Y\subset \mathbb R$ and an isomorphism $F: \mathcal Q \to \mathcal Q_Y$ such that  $ F[\mathcal I] =\mathcal B_Y \cup \mathcal C_Y \cup \mathcal D_Y$. Undoubtedly a necessary condition is that    if $I$ and $J$ belong to $\mathcal I$, then there exists $K$ in $\mathcal G$ lying between $I$ and $J$, in other words $\mathcal G$ interlaces $\mathcal I$. 

  Recall the Dedekind-completion method for  an ordered space  
$(\mathcal Q, <)$. A cut $(A,B)$ is a partition of $\mathcal Q$ into two sets $A$ and $B$, such that each element of $A$ is less than every element of $B$, and $A$ contains no greatest element.  If $B$ has the smallest element, then   $(A,B)$ corresponds to this element, 
i.e., if $b=\inf B$, then $b=(A,B)$. Otherwise, $(A,B)$ defines the unique element such that  $A$ consists of  elements  less than $(A,B)$, and $B$ consists of   elements  greater than  $(A,B)$. The family $\mathfrak C$,  which consists of all cuts $(A,B)$, is  ordered by the relation $$ (A,B)<(D,E) \Leftrightarrow A\subset D.$$ 
 Cantor's isomorphism theorem states that any countable dense unbounded, i.e., with no endpoint,  ordered spaces are isomorphic, hence the rational numbers belonging to the interval $(0,1)$ are isomorphic to any countable and unbounded densely ordered space. Therefore, if $(\mathcal Q, <)$ is countable   dense and unbounded, then $(\mathfrak C, <)$ is isomorphic to the unit interval $[0,1]$. Under such an isomorphism  $(\emptyset, \mathcal Q)$,  $(\mathcal Q, \emptyset)$
and $\mathcal Q$ pass to $0$, $1$ and the rational numbers from $(0,1)$, respectively. 

   \begin{thm}\label{tr1}
 Let $(\mathcal  G \cup \mathcal I, <)$ be a countable  ordered space such that  the sets   $ \mathcal G$ and  $\mathcal I $ are disjoint. If $ \mathcal G$ interlaces  $\mathcal I $, and  the smallest and the largest element of $\mathcal  G \cup \mathcal I$ belong to $\mathcal G $, then there exist a compact subset $Y\subset \mathbb R$ and the isomorphism  $F\colon \mathcal G\cup \mathcal I \to \mathcal Q_Y$ such that  $F[\mathcal G] =\mathcal G_Y$ and 
$ F[\mathcal I] =\mathcal I_Y$.  
   \end{thm}
   \begin{proof}  Give $(\mathcal G\cup\mathcal I)\times\mathbb Q$ the lexicographic order, where $\mathbb Q$ denotes the set of all rational numbers. Clearly, this order is countable, dense and has no endpoints, hence its Dedekind completion $\mathfrak C$ is isomorphic to $[0,1]$. Let $H\colon \mathfrak C \to [0,1]$ be such an isomorphism.
For each $I\in \mathcal I$ (resp. $G\in \mathcal G$), put  
$$
B_I=\cl_{\mathfrak C}(\{I\}\times\mathbb Q) \quad \mbox{ (resp. } B_G=\int_{\mathfrak C} \cl_\mathfrak C(\{G\}\times\mathbb Q))$$ and 
define
$$\textstyle 
Y=H\left[\mathfrak C\setminus\bigcup_{G\in\mathcal G}B_G\right].$$
   \indent  If $I\in\mathcal I$, then $H[B_I]$, being a closed interval,  is a non-trivial component of $Y$. Indeed,  if a connected subset of $Y$ met both 
	$H[B_I]$ and another $H[B_{J}]$, the interlacing hypothesis would place some $G\in\mathcal G$ between $I$ and $J$, and the removed open interval $H[B_G]$ would separate the two parts.  Thus, any connected subset of $Y$ meets at most one $H[B_I]$, where $I\in\mathcal I$.	
	Moreover, every non-trivial interval $[a,b] \subseteq Y$ has to be contained in  $H[B_I]$ for some $I\in\mathcal I$, since
	$\mathcal G$ is countable and $H^{-1}([a,b])$ has the cardinality continuum.

			If $G\in\mathcal G$, then $H[B_G]$ is a  component of $[0,1] \setminus Y$, since $H[B_G]$ is an interval with endpoints in $Y\cup \{0,1\}$.
			
			Define the necessary isomorphism   $F: \mathcal G\cup \mathcal I \to \mathcal Q_Y$  as follows.
  \[F(J)= \begin{cases}
	  (\leftarrow, \min Y),& \mbox{ for } J= \min \mathcal G,\\
		( \max Y,\rightarrow), & \mbox{ for } J=\max  \mathcal G,\\
		H[B_J],& \mbox{ for any other } J\in \mathcal G \cup \mathcal  I.
	\end{cases}\]			 
   \end{proof}

	Note that if $F$ is as in the above proof, then, for any $I\in \mathcal I$, we have
	\begin{itemize}
		\item  $F(I)\in \mathcal B_Y$ if and only if  no element of  $\mathcal G$ is neither an immediate predecessor nor an immediate successor of $I$,
		\item $F(I)\in \mathcal D_Y$ if and only if  in $\mathcal G$ there exist both the immediate predecessor and  the immediate successor of $I$,
		\item $F(I)\in \mathcal C_Y$ if and only if  in $\mathcal G$ there exists either  the immediate predecessor or the immediate successor of $I$.
		\item Also, if $G_0$ and $G_1$ are in $\mathcal G$ and $G_0$ is the immediate predecessor of $G_1$, then there exists $y \in Y$ such that $y$ is the right endpoint of the interval $G_0$ and it is also the left endpoint of the interval $G_1$.
	\end{itemize}
		    Hence, we have that if $\mathcal I_Y= \mathcal B_Y \cup \mathcal C_Y \cup \mathcal D_Y$, then the ordered space $(\mathcal G_Y\cup \mathcal I_Y, <)$  allows us to reconstruct the set $Y$ up to a homeomorphism of $\mathbb R$.

   \begin{thm}\label{tr2}
 Let $X$ and $Y$ be compact subsets of $\mathbb R$ such that there exists an isomorphism $f\colon \mathcal Q_X \to \mathcal Q_Y$  satisfying  $f[\mathcal G_X] = \mathcal G_Y$. Then  there exists an increasing homeomorphism $F\colon \mathbb R \to \mathbb R$ so that $F[X]=Y$.
  \end{thm}
  \begin{proof} 
If $x\in (-\infty, \inf X)$, then put $F(x)=x-\inf X + \inf Y$. If $x\in (\sup X, +\infty)$, then put $F(x)=x-\sup X + \sup Y$.
If $x\in (a, b)=G\in \mathcal G_X$ or $x\in [a, b]=I\in \mathcal I_X$, then put   $$F(x)=\frac{x-a}{b-a}(d-c) + c, $$ where $f(G)=(c,d)$ or $f(I)=[c,d]$. 
But, if $x\in \mathbb R \setminus\bigcup \mathcal Q_X$, then put $$ F(x)= \sup \bigcup \{f(J)\colon J \in \mathcal Q_X\mbox{ and } J < x\}.$$ By the definition,   if $I\in \mathcal I_X$, then $f(I)=F[I]\in \mathcal I_Y$ and if $G\in \mathcal G_X$, then $f(G)=F[G]\in \mathcal G_Y$. As $f$ is an isomorphism,  then any cut $(A,B)$ with respect to $\mathcal Q_X$, is mutually uniquely assigned to the cut $(f[A],f[B])$ with respect to $\mathcal Q_Y$. This implies that  $F\colon \mathbb R \to \mathbb R$ is an increasing homeomorphism and $F[X]=Y$. 
  \end{proof}
	
	\section{Some auxiliary  lemmas}
A following alternative form of the Cantor isomorphism theorem  seems to  be helpful. 
  \begin{lem}\label{le2}
	Let $\mathcal Q$ be the set of all rational numbers ordered by the   natural inequality $<$. Assume that  
$$\mathcal Q = A_0\cup A_1\cup\ldots =  B_0\cup B_1\cup \ldots, $$ where sets $A_n$ are pairwise disjoint and sets $B_n$ are also pairwise disjoint. 
If each of the sets $A_n$ and $B_n$ interlaces $\mathcal Q$, then there exists an isomorphism $f\colon \mathcal Q \to \mathcal Q$ such that $f[A_n]=B_n$ for each $n\geqslant 0$.
  \end{lem}
  \begin{proof} The set of rational numbers is countable, so   we  identify the rational numbers with the natural numbers, which  shortens the notation when we adopt the back-and-forth method. \\ \indent \texttt{Step $0.$} If $0\in A_k$, then let $f(0)\in B_k$. If $0\in B_m$, 
then choose $f^{-1}(0)\in A_m $, so that the function $\{(0,f(0)), (f^{-1}(0), 0)\}$ is increasing. \\ \indent \texttt{Step $n.$} Suppose that the image $f[ \textbf{n}]$ and the pre-image
 $f^{-1}[\textbf{n}]$, where $\textbf{n}=\{0,1,\ldots,n-1\}$, are already defined.  If $f(n)$ is not defined, then there  is $k$ such that  $n\in A_k$. In that case, since $B_k$ interlaces $\mathcal Q$, choose $f(n)\in B_k \setminus (f[\textbf{n}]\cup \textbf{n})$, so that the part of the function $f$ that has been defined so far is an increasing function. In turn, if $f^{-1}(n)$ is not defined, then there is $k$ such that 
$n\in B_k$. Since $A_k$ interlaces  $\mathcal Q$, choose $f^{-1}(n)\in A_k \setminus (f^{-1}[\textbf{n}]\cup \textbf{n+1})$, so that the part of the function $f$ that has been defined so far is an increasing function. 

By the principle of induction, the function $f\colon \mathcal Q \to \mathcal Q$ is an isomorphism such that $f[A_k]= B_k$ for each $k$. 
  \end{proof}
Due to the Cantor isomorphism theorem, Lemma \ref{le2} is also fulfilled when we substitute a countable, densely ordered space without the smallest and the largest element for the rational numbers. In other words, it is enough to assume  that a non-empty  ordered space $(\mathcal Q,<)$   is countable and unbounded,  and $\mathcal Q$ interlaces $\mathcal Q$.

Since the two arrows space, see \cite[3.10.C]{enge}, equipped with the order topology is compact and separable, we need a criterion for embedding  of compact spaces equipped with a topology determined by a lexicographic order into $\mathbb R$. When we are considering an abstract ordered space $(X,<)$, then  we  need 
a  notion to replace $\mathcal G_X$. Namely, let  $$\mathcal G(X) =\{x\in X\colon \mathop{\exists}_{y\in X} x\not= y   \mbox{ and  } (x,y)\cup (y,x)=\emptyset \}. $$
 \begin{lem}\label{le3}
Let  $(X,<)$ be an ordered space and let a countable subset $Y \subseteq X$  interlace $X\setminus Y$. If  $\mathcal G(X) \subseteq Y$, then   $(X,<)$ is isomorphic  to a  subset of  the unit interval $[0,1]$. 
  \end{lem}
  \begin{proof}
	If there exists $\sup X$, then we  enlarge $Y$ by adding $\sup X$ to it, i.e., we can assume that $\sup X \in Y$. Fix an increasing injection  $f\colon Y \to [0,1]\cap \mathbb Q$, where $\mathbb Q$ denotes the set of rational numbers. If $x\in X$, then 
let $$  B_x^{*} = \{ q\in \mathbb Q\colon \exists_{y\in Y} \; x\leqslant y \mbox{ and }  f(y)\leqslant q \}. $$ We  say that $y\in Y$ witnesses that  $q\in B^*_x$, if $x\leqslant y$ and $f(y) \leqslant q$. For example, $\sup X$ witnesses $q\in B^*_x$ for any $q\in \mathbb Q\cap [1,\rightarrow)$. But when $q<0$, then  no $y\in Y$ witnesses that $q\in B^*_x$, i.e., $q\notin B^*_x$. Let
$$B_x=\begin{cases} B^*_x \cup \{ \inf B^*_x\}, & \mbox{ if } \inf B^*_x \in \mathbb Q,\\
B^*_x,& \mbox { else } \end{cases}  $$ and put $A_x=\mathbb Q\setminus B_x.$ There is no greatest element in $A_x$, and also $B^*_x \ni q_1 < q_2 \in \mathbb Q$ implies $q_2 \in B^*_x$, so any $(A_x, B_x)$ is a  cut. When $X\ni x_1 <x_2 \in X$ and $y\in Y$ witnesses that $q\in B^*_{x_2}$, then $y$ also witnesses that $q\in B^*_{x_1}$, hence $B^*_{x_2} \subseteq B^*_{x_1}$.

We assume $\mathcal G(X) \subseteq Y$, so after a simple check we get that if $X \ni x_1  < x_2 \in X$, then there exists 
 $p\in Y$ such  that  $ x_1 \leqslant p <  x_2 $. Then $p$ witnesses that $f(p) \in B^*_{x_1}$. Let's assume that $t\in Y$ witnesses that $f(p) \in B^*_{x_2} $, i.e., $x_2 \leqslant t$ and $f(t) \leqslant f(p)$.   On the other hand, $f$ is increasing and $p < x_2 \leqslant t$, i.e., $f(p) < f(t)$,  which leads to a contradiction that shows that $x\mapsto (A_x,B_x)$ is a desired isomorphism.  
  \end{proof}
Not every isomorphism from an ordered space into the interval $[0,1]$ is continuous with respect to the order topology and the  topology hereditary from  $[0,1]$, so Lemma \ref{le3} requires strengthening.  
 \begin{pro}\label{po4} 
If $(X,<)$ is an ordered space,  a countable subset $Y \subseteq X$ interlaces $X\setminus Y$ and $\mathcal G(X)\subseteq Y$, then  there exists an increasing homeomorphism $\phi\colon X \to \phi[X] \subseteq [0,1]$.    
  \end{pro}
\begin{proof} Let $f\colon X \to [0,1]$ be an increasing injection as in Lemma \ref{le3}. Arrange the points of discontinuity $f$ into a sequence $a_0, a_1, \ldots $ and  put $X_0=f[X]$ and $\phi_0=f.$

Assume that isomorphisms $\phi_0, \phi_1, \ldots, \phi_n$ are already defined such that the points of discontinuity of $\phi_n\colon X \to \phi_n[X]=X_{n}$ are  the points  $ a_{n+k}$, where $0\leqslant k$.  Put 
   $p_n=\sup \phi_n [ (\leftarrow, a_{n})]$, $q_n=\inf \phi_n[ ( a_{n}, \rightarrow)]$ and 
$$\phi_{n+1}(x) =\begin{cases} 
\phi_n(x) + \phi_n(a_n) - p_n, & \mbox{ if } \phi_n(x)<\phi_n(a_n),\\
\phi_n(x) + \phi_n(a_n) - q_n, & \mbox{ if } \phi_n(x)>\phi_n(a_n),\\
\phi_n(a_n),& \mbox { if } \phi_n(x)=\phi_n(a_n).\end{cases}  $$
As a result, $\phi_{n+1}\colon X \to X_{n+1}=\phi_{n+1}[X]$ is an isomorphism with points of discontinuity $a_{n+k}$, where $0 < k$. 

Thus $\phi= \lim\limits_{n\to +\infty} \phi_n$ is a continuous isomorphism. Since the inverse of any isomorphism into the interval $[0,1]$ is  continuous, then $\phi$ is as we need. \end{proof}
 
\section{Tame sets in $\mathbb R$}

We say that a set $A\subseteq \mathbb R$ is \textit{tame} in $\mathbb R$ (in short a $ t $-set),  when, for any copy $Y\subseteq \mathbb R$ of $A$,  there exists a homeomorphism  $f\colon \mathbb R \to \mathbb R$   
such that $f[A]=Y$.   In fact, being a $t$-set can be considered as a variant of homogeneity of $\mathbb R$ which is similar to the variants considered in \cite{ompa}.
Clearly, any  finite subset of $\mathbb R$ is  a $ t $-set.   
The same applies to   unions of  finitely many non-trivial   closed intervals and the union of an closed interval with an isolated point.  

If $A\subset \mathbb R$ consists of two isolated points and a closed interval, then $A$ is not a $ t $-set. If $B\subset \mathbb R$ consists of two disjoint closed intervals and an  isolated point, then $B$ is not a $ t $-set. If a compact set $K\subset \mathbb R$ has exactly one limit point, then it is not a $ t $-set, as it both has copies that are  monotone and not monotone sequences with their limits. It is also easy to check that when $X \subset \mathbb R$ is compact,  contains exactly one isolated point,  contains a limit point, and $\mathcal I_X=\emptyset $, then $X$ is not a $ t $-set.

\begin{pro}\label{po6} Every $ t $-set is compact. \end{pro}
\begin{proof} Any homeomorphism of the real  numbers $\mathbb R$  transforms bounded sets onto bounded sets. But any non-compact subset 
of $\mathbb R$ has a bounded copy, as well a copy which is not bounded. These two properties rule out that a non-compact subset of $\mathbb R$ is a $ t $-set. 
 \end{proof}

The next proposition is essentially an alternative definition of tame sets in $\mathbb R$.

\begin{pro}\label{co7} A compact set $X\subset \mathbb R$ is a $ t $-set if and only if, for every copy $Y\subseteq \mathbb R $ of $X$, there exists a monotone bijection $f\colon \mathcal Q_X \to \mathcal Q_Y$ such that $f[\mathcal I_X]=\mathcal I_Y$. \end{pro}
\begin{proof}  
When  $F\colon \mathbb R \to \mathbb R$ is a homeomorphism such that $F[X]=Y$, then let $f(I)=F[I]$ for each $I\in \mathcal Q_X$. Obviously,  $f\colon \mathcal Q_X \to \mathcal Q_Y$ is a monotone bijection  and $f[\mathcal I_X]=\mathcal I_Y$. 

When $f\colon \mathcal Q_X \to \mathcal Q_Y$ is an increasing (resp. decreasing) bijection  such that $f[\mathcal I_X]=\mathcal I_Y$,  then 
let $F|_I$ be an increasing (resp. decreasing) bijection  from $I$ onto $f(I)$ for each $I \in \mathcal Q_X$. So $\bigcup \{ F|_I \colon I \in \mathcal Q_X\}$ is a monotonic bijection between dense subsets of $ \mathbb R $, which can be uniquely extended to a homeomorphism 
$F\colon \mathbb R \to \mathbb R$   such that $F[X]=Y$.
 \end{proof}

We will call the mapping $$f(x)=\begin{cases} a+b-x, & \mbox{ if } x\in [a,b];\\
x,& \mbox { else } \end{cases}  $$  the \textit{rotation} of the interval $[a,b]$. Note that, if $X\subset \mathbb R$ is a compact set and $\{a,b\} \cap X =\emptyset$, then $f|_X \colon X \to f[X]$ is a homeomorphism such that if  $J\in \mathcal C_{X\cap [a,b]}$ has  the left (resp. right) endpoint free, then the interval $f[J]\in \mathcal C_{f[X]}$ has  the right (resp. left) endpoint free. The next lemma  specifies the content of \cite[Remark 8]{bkp}.

\begin{lem}\label{le8} Suppose $X\subset \mathbb R$ is a compact set and $\emptyset \not= \mathcal C \subseteq \mathcal C_X$. Then there exists a homeomorphism $f\colon X \to f[X] \subset \mathbb R$ such that the image $f[J]$ is an $f[X]$-interval with the left endpoint free if and only if $J\in \mathcal C$.
\end{lem}
\begin{proof} Let $J_0, J_1, \ldots$ be a sequence of all $X$-intervals from $\mathcal C_X$.   
If  $J_0\in \mathcal C$ has the left endpoint free or $J_0\in \mathcal C_X \setminus \mathcal C$ has the right endpoint free, then put $f_0= \id_X$. 
Otherwise,  choose numbers $q_0$ and $p_0$ not belonging to $X$ such that $J_0 \subset  (q_0,p_0)$.   Let $f_0$ be the rotation of  $[q_0,p_0]$.
Since $[q_0,p_0]\cap X$ is a closed-open subset of $X$, then    $f_0|_X\colon X \to X_0=f_0[X]$ is a homeomorphism and $f[J_0]\in \mathcal C_{X_0}$.  

Assume that the bijections $f_0, f_1, \ldots, f_{n-1}$ are already defined, and each $f_k$ is the rotation of  $[q_k,p_k]$ or the identity map.
  For the sake of brevity, let's put $f_k^*=f_k \circ \ldots \circ f_0$ and $X_k=f_k^*[X]$, including that $f_k^*|_X \colon X \to X_k$ is a homeomorphism. Let $f_n= \id_X$ whenever 
\begin{itemize}
	\item[] $ J_n \in \mathcal C$ and the interval $f_{n-1}^*[J_n]$ has the left endpoint free in $X_{n-1}$, 
	\item[] $ J_n \in \mathcal C_X \setminus \mathcal C$ and the interval $f_{n-1}^*[J_n]$ has the right endpoint free in $X_{n-1}$.
\end{itemize}
  Otherwise, choose numbers $q_n$ and $p_n$  belonging to $\mathbb R\setminus X_{n-1}$ such that 
	$$f^*_{n-1}[J_n]\subset (q_n,p_n) \mbox{ and } p_n - q_n< 2\delta (J_n),$$ where $\delta (J_n)=\sup J_n -\inf J_n$,  and   the closed interval $[q_n,p_n]$ does not intersect the union 
$$\bigcup \{f^*_{n-1}[J_k]\colon 0\leq k <n\}\cup \bigcup \{\{ q_k,p_k\}\colon k<n \mbox{ and } f_k \mbox{ is a rotation} \}.$$  Then let $f_n$ be the rotation of  $[q_n,p_n]$.
 Since $\sum\limits^\infty_{n=0} \delta (J_n) < \sup X - \inf X$ and $$|f^*_{n-1}(x) - f^*_{n}(x)|<2\delta(J_n),  \mbox{ for each } x\in X, $$ then the sequence of homeomorphisms $f^*_n|_X$ is a Cauchy sequence with respect to  the supremum norm, so it converges uniformly to the  function $f\colon X \to f[X]$ which is continuous.
Let $x$ and $y$ be distinct points in $X$. If $|x-y| = |f^*_n(x) -f^*_n(y)|$ for each $n$, then $f(x)\not=f(y)$.  But if $n$ is the first number such that $|x-y| \not= |f^*_n(x) -f^*_n(y)|$, then $f_n$ has to be a rotation. Without loss of  generality, we can assume that $f^*_n(x) <q_n <f^*_n(y).$ By the definition, we get  
$f(x) <q_n <f(y).$ Thus $f$ is an injection, consequently it is a desired homeomorphism. 
\end{proof}

It seems that the following lemma may be useful..

\begin{lem}\label{le9} Let $X\subseteq \mathbb R $ be a compact set and $\{ x\}$ be a trivial component of $X$. Then there exist homeomorphisms $f_1\colon X\to X_1$ and $f_2\colon X\to X_2$ such that $f_1(x) = \inf X_1$ and $f_2(x) = \sup X_2$. \end{lem}
\begin{proof} If $x\in \{\inf X, \sup X\}$, then functions $x\mapsto x$ and $x\mapsto -x$ work. When $\{x\}$ is a closed-open set, then by shifting the isolated point $x\in X$  to the left or right we get the necessary homeomorphisms. In the other case, fix an increasing sequence $(b_n)$ and a decreasing sequence $(c_n)$, assuming  $b_0=-\infty$, that converge to $x$ and are contained in the complement of $X$. Let $(J_n)$ be a sequence of pairwise disjoint closed intervals contained in the complement of $X$ such that $c_n \in J_n$ for all $n$. 
Let $f_1\colon X\to f_1[X]$ on the set $[x, +\infty) \cap X$ be the identity and let $f_1$ shrink suitably and move each closed-open  set $(b_n, b_{n+1}) \cap X$ into $J_n$. Since the sets $(b_n,b_{n+1}) \cap X$ are closed-open in $X$ and the intervals $J_n$ are  disjoint with $X$, then $f_1\colon X \to f_1[X]$ is a homeomorphism such that 
$f_1(x)= \inf  f_1 [X] = x$. 
Putting $ f_2=-f_1$, we get the homeomorphism $f_2$ such that $f_2(x) =\sup  f_2 [X] = -x$. 
\end {proof}

\begin{cor}\label{c10} If $X\subseteq \mathbb R $ is a compact set such that   $\{ x\}$ and $\{y\}$ are different  trivial components, then  there exists a homeomorphism $f\colon X\to f[X]$  such that $f(x) = \inf f[X]$ and $f(y) = \sup f[X]$.  \end{cor}
\begin{proof} Without loss of generality, we can assume that $x<y$. Let $z\in \mathbb R \setminus X$ be such that $x<z<y$. Applying Lemma \ref{le9} to both closed-open sets 
$X_1=X \cap (-\infty,z)$ and  $X_2=X \cap (z,+\infty)$, we get a homeomorphism $f_1$  which carries $ X_1$ into  $[x,z)$ and a homeomorphism $f_2$ which carries 
$ X_2$ into  $(z,y]$. Additionally, we can assume that  $f_1(x) = \inf f_1[X_1]$ and  $f_2(y) = \sup f_2[X_2]$, so  $f_1 \cup f_2$ makes a desired homeomorphism $f$. \end{proof}

\begin{pro}\label{p11} If a set $X\subset \mathbb R$ is compact and  $\mathcal C_X\not= \emptyset$, then $X$ is not a $ t $-set.
\end{pro}
\begin{proof}  Suppose that $\mathcal C_X=\{[a,b]\} $. When $a$ (resp. $b$) is the free endpoint, then  $V=X\cap [a,\rightarrow)$ (resp. $V=X\cap (\leftarrow, b]$) is a closed-open set in $X$. Create $X_1$ by shifting 
 $V $ to the left of $\inf X$ (resp. to the right of $\sup X$), leaving the rest untouched.    Since $\mathcal G_X$ is infinite, we can divide $X$ into three disjoint and non-empty closed-open subsets, then  create $X_2$ by shifting these sets so that the interval $[a,b]$ is contained in the middle set. Clearly, no homeomorphism of $\mathbb R$ carries $X_1$ onto $X_2$.

Now assume that $\mathcal C_X$ contains at least two intervals. Using Lemma \ref{le8}, compare \cite[Remark 8]{bkp}, create a copy $X_1$ of $X$ such that each interval in $\mathcal C_{X_1}$ has the left endpoint free. Then create a copy $X_2$ of $X$ such that among  intervals in  $\mathcal C_{X_2}$ there are an interval with  the left endpoint free and an interval with the right endpoint free. Clearly, $X_1$ and $X_2$ witness that $X$ is not a $t$-set. 
\end{proof}

 As noted in \cite[pp. 8--9]{bkp},  the Cantor set and  an $M$-Cantorval  are  $ t $-sets. A common feature of these sets is that they can be distinguished by  properties of countable ordered spaces. 
  Let us describe these properties in a little more detail than in \cite{bkp}. 
A set $\mathbb C \subset \mathbb R$ is a \textit{copy of the Cantor set} whenever it is compact, dense-in-itself and has a base which consists of closed-open sets, compare \cite[6.2.A.(c)]{enge}. 
\begin{pro}\label{p12} Let $X\subseteq \mathbb R$ be a nonempty compact set. $X$  is a copy of the Cantor set if and only if  $\mathcal G_X=\mathcal Q_X$ interlaces $\mathcal Q_X$.
\end{pro}
\begin{proof} If $X\not=\emptyset$, then $\mathcal G_X$ has at least two points. Assuming that $\mathcal G_X=\mathcal Q_X$ interlaces $\mathcal G_X$,  we get that $\mathcal G_X$ is infinite, and also that there are no isolated points in $X$. Since $\mathcal I_X=\emptyset$, then $X$ is  compact,  nowhere dense in $\mathbb R$ and has no isolated point, so  $X\not=\emptyset$ is a copy of the Cantor set.  

Conversely, if $X\subset \mathbb R$ is a copy of the Cantor set, then $\mathcal I_X =\emptyset $, i.e., $\mathcal G_X=\mathcal Q_X$. Also between any two different intervals from $\mathcal G_X$ lies a non-empty open subset of $X$, son there also lies an interval from $\mathcal G_X$, this shows that $\mathcal Q_X=\mathcal G_X$ interlaces $\mathcal Q_X$.
\end{proof}

 If a subset $\mathbb C \subset \mathbb R $  is non-empty, bounded,  regularly closed and
 endpoints of any non-trivial component of $\mathbb C$ are limit points
of trivial  components of $\mathbb C$, then it is called an $M$-\textit{Cantorval}, compare \cite[Definition 13]{nit2}.
The name Cantorval has been introduced by P. Mendes and F. Oliveira \cite{mo} as:  \textit{A perfect subset of $\mathbb R$, such that any gap is accumulated on each side by
infinitely many intervals and gaps is called an $M$-Cantorval}. However, the intervals $(-\infty, \inf X)$ and $(\sup X, +\infty)$ are not included in the gaps. Refine this definition by adding that $\inf X$ and $\sup X$ are also accumulated by infinitely many gaps. There are several characterizations of $M$-Cantorvals known in the literature, for example see
 J. A. Guthrie and J. E. Nymann \cite{gn}. The following two propositions add another characterization of $M$-Cantorvals

\begin{pro}\label{p13} If  $X$ is an $M$-Cantorval, then  $\mathcal B_X$ interlaces $ \mathcal Q_X$.   
\end{pro}
\begin{proof} If $X$ is an $M$-Cantorval, then no $X$-interval has a free endpoint, i.e., $\mathcal B_X=\mathcal I_X$,  and consequently between any $X$-interval and any $X$-gap lies an $X$-interval. Any $M$-Cantorval is regularly closed, so between any two $X$-gaps lies an $X$-interval, hence  $\mathcal B_X$ interlaces $ \mathcal Q_X$.
\end{proof}

\begin{pro}\label{p14} If  $X\not=\emptyset$ is a compact subset of real  numbers such that $\mathcal B_X$ interlaces $ \mathcal Q_X$, then $X$ is an $M$-Cantorval.   
\end{pro}
\begin{proof} The family $\mathcal B_X$ has to be infinite and $\mathcal C_X \cup \mathcal D_X = \emptyset$, since $X\not=\emptyset$ and $\mathcal B_X$ interlaces $ \mathcal Q_X$. For the same reasons, no point in $X$ is isolated and $X$ does not contain a copy of the Cantor set as a closed-open set. Thus  $X$ is a regularly closed subset of $\mathbb R$. Since $\mathcal I_X =\mathcal B_X$, endpoints of each bounded $X$-gap has be to accumulated on each side by infinitely many intervals from $\mathcal B_X$, and therefore also through infinitely many $X$-gaps. Since $(-\infty, \inf X )$ and $(\sup X, +\infty)$ are considered $X$-gaps, then both $\inf X$and  $\sup X$ cannot be the endpoint of any $X$-interval. All the above-mentioned properties are sufficient to establish that $X$ is an $M$-Cantorval.
\end{proof}

Note that, if $\emptyset\not=X\subset \mathbb R$ and $X$ is  compact, and  $\mathcal B_X$ interlaces $ \mathcal Q_X$, then families $\mathcal B_X$ and $\mathcal G_X$ are dense in  $ \mathcal Q_X$. Indeed, by the definition, there is always an $X$-gap between any two distinct $X$-intervals. Since $\mathcal B_X$ interlaces $ \mathcal Q_X$,  there is always an $X$-interval between an $X$-gap and an $X$-intervals, and also, there is always an $X$-interval between any two distinct $X$-gaps. So, similarly as in \cite[Proposition 10, Lemma 11 and Corollary 12]{bkp} we get that any  $M$-Cantorval is a  $ t $-set. 

\begin{cor}\label{c15}  The Cantor set and each $M$-Cantorval  are  $ t $-sets.   
\end{cor}
\begin{proof} If $X$ and $Y$ are copies of the Cantor set, then $\mathcal Q_X=\mathcal G_X$ and $\mathcal Q_Y=\mathcal G_Y$ are dense. If $f\colon \mathcal Q_X \to \mathcal Q_Y$ is an isomorphism, then Theorem \ref{tr2} works. 

Now, suppose that $X$ and $Y$ are $M$-Cantorvals. Let $h\colon \mathcal Q_X \to \mathcal Q$ and $g\colon \mathcal Q_Y \to \mathcal Q$ be isomorphisms. Each of the sets $h[\mathcal G_X]$, $h[\mathcal B_X]$, $g[\mathcal G_Y]$  and $g[\mathcal B_Y]$  interlaces 
$\mathcal Q$.  By Lemma \ref{le2}, there exists an isomorphism $f\colon \mathcal Q \to \mathcal Q$ such that $f(h[\mathcal G_X]) = g[\mathcal G_Y]$. 
Thus $g^{-1} \circ f \circ h\colon \mathcal Q_X \to \mathcal Q_Y$ is an isomorphism such that $(g^{-1} \circ f \circ h)[\mathcal G_X] = \mathcal G_Y$, so once again Theorem \ref{tr2} works. \end{proof}

As noted in \cite{bkp}, Corollary \ref{c15} coincides with one of the  results presented  by M. Walczy\'nska \cite{wal} in her doctoral dissertation.

\begin{pro}\label{p16} Let $X\subseteq \mathbb R$ be a compact set and $\mathcal I_X=\emptyset $. Then $X$  is a $ t $-set if and only if $X$ is finite or $X$ is a copy of the Cantor set. 
\end{pro} 
\begin{proof} Suppose that $X$ is a $t$-set. If $X$ is  discrete, then $X$, being a compact set,  must be finite.  If $X\not=\emptyset$ contains no isolated point,   then $X$ is a dense-in-itself zero-dimensional compact metric space, since there is no $X$-interval. Therefore  $X$ is a copy of the Cantor set, see \cite[6.2.A. (c)]{enge}. 
When $X$ contains a limit point and $a\in X$ is the only isolated point  in $X$, then $X\setminus \{a\}$ is   a copy of the Cantor set. Since it can be that $a=\inf X$ and it also  can be that $a\notin \{\inf X, \sup X\}$, so $X$  is not a $ t $-set.
In turn, if $X$ contains at least two isolated points and a limit point $b$, then   it can be that $b=\inf X$, because of  Lemma \ref{le9},  and  it also can be that $ \{\inf X, \sup X\}$ consists of isolated points by  Corollary  \ref{c10}. Therefore   $X$ is not a $ t $-set.

Obviously, any finite set is a $t$-set. By Corollary \ref{c15}, the Cantor set is a $t$-set. So, the reverse implication is also fulfilled. 
\end{proof}

\begin{pro}\label{p17} Let $X\subseteq \mathbb R$ be a compact set  such that $\mathcal I_X=\mathcal D_X \not=\emptyset $. Then $X$ is a $ t $-set if and only if it is a union of a closed interval and an isolated point or it is a union of finitely many closed intervals. 
\end{pro} 
\begin{proof} If the family $\mathcal D_X$ is infinite, then $X$ contains  a trivial component $\{ b\}$.  By Lemma \ref{le9}, it can be that $b=\inf X$.
It also can be that $ \{\inf X, \sup X\}$ consists of endpoints of $X$-intervals.  So, $X$ is not a $ t $-set. 

If $X$ is a $t$-set and  $\mathcal I_X=\mathcal D_X$ is finite, then $X$ contains at most  one isolated point. Indeed, if $b$ is the isolated point in $X$, then $X$ has to be  a union of a closed interval and $b$. In turn, if $X$ contains no isolated point, then $X \setminus \bigcup \mathcal I_X$ is a compact and dense it itself subset. Thus, if $X = \bigcup \mathcal I_X$, then $X$ is a union of finitely many closed intervals. But if $X \setminus \bigcup \mathcal I_X$  is a copy of the Cantor set, then $X$ is not a $ t $-set.

The reverse implication is obvious.\end{proof}

\begin{pro}\label{p18} A regularly closed and compact set $X\subseteq \mathbb R$, such that the family $\mathcal I_X=\mathcal B_X  $ is infinite, is a $ t $-set if and only if $X$ is an $M$-Cantorval.
\end{pro}
\begin{proof} Suppose that $X$ is a regularly closed $t$-set and  $\mathcal I_X=\mathcal B_X  $ is infinite. Then the set $\{ \inf \bigcup \mathcal I_X, \sup \bigcup \mathcal I_X\}$ consists of limit 
points such that $\{\inf\bigcup \mathcal I_X\}$ and $\{ \sup \bigcup \mathcal I_X\}$ are trivial components. By Corollary  \ref{c10}, there is a copy of $X$ such that $\inf\bigcup \mathcal I_X= \inf X$ and $\sup\bigcup \mathcal I_X= \sup X$. For this reason, if $X$ has an isolated point or $X$ contains a closed-open subset that is a copy of the Cantor set, then $X$ is not a $ t $-set. Hence  we get $\cl_X (\bigcup \mathcal I_X) =X$. Since $\mathcal C_X=\emptyset= \mathcal D_X$,  we conclude that  $X$ is an $M$-Cantorval.

Corollary \ref{c15} gives the inverse implication.
\end{proof}

\section{When $\mathcal B_X $ is  finite} 
Consider the following  compact subsets of $\mathbb R$.
  \begin{itemize}
	\item  $\mathbb A = [0,1] \cup \{\frac{-1}{n}\colon n>0\} \cup \{\frac{n+1}{n}\colon n>0\}$, 
	\item  $\mathbb B = [0,1] \cup \bigcup \{[\frac{-1}{2n}, \frac{-1}{2n+1}]\colon n>0\} \cup \bigcup \{[\frac{2n+2}{2n+1}, \frac{2n+1}{2n}]\colon n>0\}$, 
	\item  $\mathbb D= [0,1] \cup \mathbb C_0 \cup \mathbb C_1,$ where $\mathbb C_0$ and $ \mathbb C_1$ are copies of the Cantor set such that $\sup \mathbb C_0 =0$ and $\inf  \mathbb C_1 =1$.
   \end{itemize} 
Each of these sets is a $t$-set. Indeed, when $Y$ is a copy of one of them, then $\mathcal B_Y=\{J\}$. Each endpoint of $J$ is accumulated by a monotone sequence of $Y$-gaps such that between adjacent gaps lies an isolated point, applies to $\mathbb A$, or a $Y$-interval with free endpoints, applies to $\mathbb B$, or a  copy of the Cantor set, applies to $\mathbb D$. This is enough to define a homeomorphism of $\mathbb R$ that carries $Y$ onto $\mathbb A$ or $\mathbb B$, or $\mathbb D$.

\begin{pro}\label{p19} 
If  $X\subseteq \mathbb R$ is a $ t $-set  and $\mathcal B_X\not=\emptyset$ is finite, then  $X$ is one of the following
\begin{enumerate}
	 	\item A union of $|\mathcal B_X|$ many pairwise disjoint copies of  $\mathbb A$, 
	\item A union of $|\mathcal B_X|$ many pairwise disjoint copies of  $\mathbb B$,
	\item A  union of $|\mathcal B_X|$ many pairwise disjoint copies of $\mathbb D$.
	\end{enumerate}
\end{pro}
\begin{proof} Suppose that $X\subseteq \mathbb R$ is a $ t $-set  and $\mathcal B_X\not=\emptyset$ is finite. Then endpoints of each $I\in \mathcal B_X$ are accumulated  either by isolated points or by intervals from $\mathcal D_X$, or by closed-open copies of the Cantor set.   
 If $X$ contains an isolated point, then $X$ has to contain infinitely many isolated points and $\mathcal I_X= \mathcal B_X$, and $X$ contains no  closed-open copy of the Cantor set. Thus, $X$ has to be  a union of finitely many pairwise disjoint copies of  $\mathbb A$. To put it precisely, $X$  is a union of 
$|\mathcal B_X|$ many copies of $\mathbb A$.  The proofs for the remaining cases are analogous, i.e.,  if $X\setminus \bigcup \mathcal B_X$ contains a non-trivial interval, then $X$ has to be as in the case  (2); but  if  $X\setminus \bigcup \mathcal I_X$ contains a limit point, then $X$ has to contain a closed-open copy of the Cantor set, so $X$ is as in the case  (3).
\end{proof}

In the proof above, and in several proofs given below, we deliberately omit certain details because they are common knowledge or would constitute repetition.

\section{When  $\mathcal B_X$ is scattered and infinite}  
We are left to discuss other  $ t $-sets $X$, when $\mathcal B_X$ is infinite. In Proposition  \ref{p14},
  a new characterization of $M$-Cantorvals  is given. Let's add another one, which uses the order topology.

\begin{pro}\label{p20}
 If  $X$ is a $ t $-set and the order topology on $\mathcal Q_X$  is not scattered, then  $X$ is  an $M$-Cantorval or there exists a  closed-open subset of $X$ which is a copy of the Cantor set. \end{pro}
\begin{proof} When $\mathcal Q_X$ contains a  subset homeomorphic to the rational numbers, i.e., $\mathcal Q_X$ is not scattered, then there exist  continuum many trivial components in $X$.  So, the set $\{\inf X, \sup X\}$ could be populated by two trivial components. Hence   $\mathcal D_X$ has to be  empty and   there is no isolated point in $X$, since $X$ is a $t$-set. But if  no  closed-open subset of $X$  is a copy of the Cantor set, then $X$ has to be regularly closed and then Proposition \ref{p18} works, i.e., $X$ is an $M$-Cantorval.
 \end{proof}

\begin{cor}\label{c21} If  $X$ is a $ t $-set and $\mathcal B_X$ is not scattered, then $X$ is an $M$-Cantorval. \end{cor}
\begin{proof} Since $\mathcal B_X$ is not scattered, there exist a trivial component $\{a\} \subset X$ such that $a \in \cl_\mathbb R(\bigcup \mathcal B_X)$.
 If $X$ contains a closed-open  copy of the Cantor set, then  $\mathbb R$ contains copies $Y$ and $Z$ of $X$ such that $$\textstyle  \inf Y \in \cl_\mathbb R (\bigcup \mathcal B_Y) \mbox{ and } \{\inf Z, \sup Z\} \cap \cl_\mathbb R (\bigcup \mathcal B_Z) = \emptyset, $$ so $X$ can  not be a $t$-set.  By the above proposition, $X$ has to be  an $M$-Cantorval. \end{proof}

Note that, if  $X$ is a $ t $-set and $\mathcal I_X$ is not scattered, then  $\mathcal B_X$ is also not scattered. Indeed, then $\mathcal C_X=\emptyset$ and $\mathcal D_X$, being  discrete, is countable, and $\mathcal I_X$ is uncountable.

\begin{pro}\label{p22} 
If  $X$ is a $ t $-set such that $\mathcal B_X$ is scattered and infinite, then  $X$ contains a closed-open set that is a copy of either $\mathbb A$ 
or $\mathbb B$, or $\mathbb D$. Moreover, $\mathcal B_X$
is compact in the topology induced from the order topology on $\mathcal Q_X$.
 \end{pro}
\begin{proof} 
Since  the order topology on $\mathcal B_X$ induced by $<$ is  scattered, there exists an $X$-interval $[a,b]\in \mathcal B_X$ that is isolated in $\mathcal B_X$ and there exist points $a_n \in X \setminus [a,b]$ such that $\lim\limits_{n\to +\infty} a_n =a$. If each point $a_n$ is isolated  in $X$, then   $X$ contains a closed-open set that is a copy of $\mathbb A$.
If the sets $\{a_n\}$ are trivial components of $X$ and no $a_n$ is an isolated point, then   $X$ contains a closed-open set that is a copy of $\mathbb D$. But if the points $a_n$ belong to $\bigcup \mathcal D_X$, then   $X$ contains a closed-open set that is a copy of $\mathbb B$.

Suppose  $\mathcal B_X$
is not compact in the topology induced from the order topology on $\mathcal Q_X$. If $x\in \cl_X \bigcup \mathcal B_X \setminus \bigcup \mathcal B_X$, then by Lemma \ref{le9}, there exists a homeomorphism   $f\colon X \to f[X] \subset  \mathbb R$   with $f(x) = \sup f[X]$. So, $\{ f(x)\}$ is a trivial component of 
$f[X]$, where $f(x)$  is neither an isolated point in $f[X]$, nor an endpoint of $f[X]$-interval.  Since  no closed-open subset of $X$ is a copy of the Cantor set. We get a contradiction with the first part of this proof, because then $X$ contains neither a copy of $\mathbb A$ nor a copy of $\mathbb B$, nor a copy of $\mathbb D$.  
   \end{proof}

Let us recall a notion of  \textit{the derived set of E of order $\alpha$, } compare \cite[p. 64]{si}. Namely, $E= E^{(0)}$ and $E^{(\alpha +1)}$ denotes the derived set of $E^{(\alpha)}$, and $E^{(\alpha )}= \bigcap\{E^{(\beta)}\colon \beta < \alpha\}$ for a limit ordinal $\alpha$. In the following we will use derived sets with respect to the order topology for $(\mathcal I_X, <)$, where $X\subset \mathbb R$ is a compact set and $\mathcal C_X=\emptyset$. 

For a compact subset $X\subset \mathbb R$ we will need to consider the following condition.

 $(\star)$ \hspace{0.1cm} \textit{If $I\in \mathcal I_X$, then there exist an ordinal  $\alpha$ and $X$-gaps $J_0, J_1 \in \mathcal G_X$ such that 
 $(\mathcal I_X \cap [J_0, I])^{(\alpha)} = (\mathcal I_X \cap [I, J_1])^{(\alpha)}=\{ I\}.$} 

\begin{lem}\label{l23}
 If  $X$ is a $ t $-set such that  $\mathcal I_X$  equipped with  the order topology is scattered, then $X$ satisfies condition $(\star)$. 
   \end{lem} 
\begin{proof} Since $X$ is assumed to be a $t$-set, then, by Proposition \ref{p11}, we have $\mathcal C_X=\emptyset$. When $I\in \mathcal D_X$, then $I$ is an isolated point with respect to order topology for $(\mathcal I_X,< )$, so $\alpha =0$ is good.
  When the family $\mathcal B_X$ is finite, then apply Proposition \ref{p19} and check that then the thesis  is fulfilled. So, we can proceed assuming  that $\mathcal B_X$ is infinite.
	
	If $ I\in \mathcal B_X$  is an  $X$-interval that does not satisfy the thesis, then   choose $J_0\in \mathcal G_X $ and $ J_1\in \mathcal G_X$ such that 
	 $$\mathcal I_X^{(\alpha)} \cap [J_0, I] = \{I\} = \mathcal I_X^{(\beta)} \cap [I, J_1] \mbox{ and } \mathcal I_X^{(\alpha)} \cap [J_1,\rightarrow) \not= \emptyset,$$
	where $0 <\alpha <\beta$ (when $\beta < \alpha $, the reasoning is symmetrical). Choose an interval $K$ which is isolated in $\mathcal I_X^{(\alpha)} \cap [I, J_1]$. Let $Z=(a,b) \cap X$ be such that $a,b \not\in X$ and $ \mathcal I_Z^{(\alpha)}=\{K\}$. Let $X_1$ be a copy of $X$ which is formed from $X$ by moving $(-\infty, \sup J_0) \cap X$ to the right of $\sup X$. Thus $\inf \mathcal I_{X_1}^{(\alpha)} =I$. Let $X_2$ be a copy of $X$ which is formed from $X_1$ by moving $Z$ to the left  of   $\inf X$. Thus $\inf \mathcal I_{X_2}^{(\alpha)} =K'$, where $K'$ is an isolated point in $\mathcal I_{X_2}^{(\alpha)}$. Suppose  $F\colon \mathbb R \to \mathbb R$ is a monotone bijection such that $F[X_1] =X_2$.  If $f(J) = F[J]$, for every $J\in \mathcal I_{X_1}$, then we get that $f\colon \mathcal I_{X_1}\to\mathcal I_{X_2}$ is a monotone bijection.  Since $\inf \mathcal I_{X_1}^{(\alpha)} =I \in  \mathcal I_{X_1}^{(\beta)}$, then $I$ is not isolated in $\mathcal I_{X_1}^{(\alpha)}$, but $K'$ is isolated in $\mathcal I_{X_2}^{(\alpha)}$. Therefore $f$ can not be increasing. Now, suppose that $f$ is decreasing. We get that the point $$f^{-1}(K')=\sup \mathcal I_{X_1}^{(\alpha)} =\sup \mathcal I_{X_2}^{(\alpha)}= f(I)$$ is both  isolated  and not isolated  in $\mathcal I_{X_1}^{(\alpha)}$; a contradiction which completes the proof.
  \end{proof}

	\section{A counterpart of Mazurkiewicz-Sierpiński theorem. }
S. Mazurkiewicz and W. Sierpiński \cite{ms} proved that a scattered, metric and compact space is homeomorphic to a countable ordinal $\omega^\alpha n +1$. A counterpart of this result is presented  below.

Let $(X,<)$ be an ordered space. Put $$ S(X)= \{\frac1n\colon n\ne 0\} \times X \cup (\{0\} \times [0,1]), $$ where $n\ne 0$ means that $n$ is a non-zero integer.
Ordering $S(X)$  lexicographically (the set $\{\frac1n\colon n\ne 0\}$ inherits the order from  $\mathbb R)$, we get an ordered space $(S(X),<)$. Wherein
  \begin{itemize}
	\item [] The space $(S(\{0\}),< )$ is isomorphic to $(\mathbb A, <)$,
	\item [] The space $(S([0,1]),< )$ is isomorphic to $(\mathbb B, <)$,
	\item [] The space $(S(\mbox{The Cantor set}),< )$ is isomorphic to $\mathbb D$,
	\item [] The space $(S(\mbox{$M$-Cantorval}),< )$ is isomorphic to an $M$-Cantorval.
  \end{itemize}
	Spaces $\mathbb A$, $\mathbb B$, $\mathbb D$ and $M$-Cantorvals are compact, so  the isomorphisms mentioned above are also homeomorphims. 

The operation $X\mapsto S(X)$ can be iterated, namely if $\alpha < \omega_1$ is not a limit ordinal, then put $S_0(X)=X$ and $S_{\alpha +1}(X) =S(S_\alpha(X))$. But if $\alpha < \omega_1$ is  a limit ordinal, then let $\alpha^*$ be a union of two copies of $\alpha$, one decreasing contained in $(0, \rightarrow)$, the other increasing contained in $(\leftarrow, 0)$. For example,  $\omega_0^*$ is a copy of  $\{\frac1n\colon n\ne0\} \subset \mathbb R$. Put $$S_\alpha (X)= \bigcup \{\{\gamma\} \times S_\gamma(X)\colon \gamma \in \alpha^*\} \cup \{0\} \times [0,1], $$   and then equip $S_{\alpha}(X)$ with the lexicographical order.

For the purposes of the next theorems, let $\textbf{X}_0=[0,1]\subset \mathbb R$ and $\textbf{X}_\alpha \subset \mathbb R$ be an isomorphic copy of $S_\alpha([0,1])$, because of Proposition \ref{po4}, we can assume that  $\textbf{X}_\alpha$ is compact.   Note that each set $\textbf{X}_\alpha = \bigcup \mathcal I_{\textbf{X}_\alpha}$ is a regularly closed subset of $\mathbb R$. When we consider $\mathcal I_{\textbf{X}_\alpha}$  with  the order topology, then it is a scattered space and its $\alpha$-th derivative  is a singleton. 

\begin{pro}\label{p24} Let $Y\subseteq \mathbb R$ be a compact and regularly closed set such that $\mathcal I_{{Y}}^{(\alpha)} = \{K\}$. If $Y$ satisfies condition $(\star)$, then $Y$ is a copy of $\mathbf{X}_\alpha$. \end{pro}

\begin{proof} When $\alpha =0 $, then  both $Y$ and $\textbf{X}_0$ are closed intervals. Assume inductively that the thesis is fulfilled for each $\beta <\alpha$.

When $\alpha = \beta+1$, then $\mathcal I^{(\alpha)}_{\textbf{X}_\alpha} = \{[0,1]\}$ and let $\mathcal I^{(\alpha)}_{Y} = \{[q,p]\}$. Choose numbers $a_n,b_n, c_n$ and $d_n$,  belonging to $\mathbb R\setminus Y$, so that 
  \begin{itemize}
	\item $a_1 < b_1 < a_2 < \ldots < b_n < a_{n+1} < \ldots < q$, 
	\item $c_1 > d_1 > c_2 > \ldots > d_n > c_{n+1} > \ldots > p$,
	\item if $n>0$, then $\mathcal I^{(\beta)}_{[a_n,b_n]\cap Y}$ (resp. $\mathcal I^{(\beta)}_{[d_n,c_n]\cap Y}$) contains  exactly one interval belonging to $\mathcal I^{(\beta)}_{Y}$.
  \end{itemize} 
	Analogously,  choose numbers belonging to $\mathbb R \setminus \textbf{X}_\alpha$, taking $0$ for $q$ and $1$ for $p$. Then fix an increasing homeomorphism $F\colon Y \to \textbf{X}_\alpha $ so that
$F[[q, p]] =[0,1]$ and  $$\{ F[I]\colon I \in \mathcal I^{(\beta)}_Y \} = \mathcal I^{(\beta)}_{\textbf{X}_\beta},$$
assuming inductively that each  restriction of $F$ to $Y\cap ([a_n, b_n] \cup [d_n, c_n]) $ is  continuous and increasing, and such that 
$$F[ Y\cap ( [a_n, b_n] \cup [d_n, c_n]) ] = \textbf{X}_\beta \cap ([F(a_n), F(b_n)] \cup [F(d_n), F(c_n)]),$$ where numbers $F(a_n), F(b_n), F(d_n)$ and $F(c_n)$ belong to $\mathbb R \setminus \textbf{X}_\alpha$,  we get that $F$ is a desired homeomorphism.

If $\alpha$ is a  limit ordinal, then let     $\mathcal I^{(\alpha)}_{Y} = \{[q,p]\}$ and, for each $\beta < \alpha$, let $ I_\beta= \inf \mathcal I^{(\beta)}_Y$ and $ J_\beta= \sup \mathcal I^{(\beta)}_Y$. Choose numbers $a_\beta,b_\beta, c_\beta$ and $d_\beta$,  belonging to $\mathbb R\setminus Y$, so that 
  \begin{itemize}
	\item[] $a_1 < b_1 < a_2 < \ldots < b_\beta < a_{\beta+1} < \ldots < q$, 
	\item[] $c_1 > d_1 > c_2 > \ldots > d_\beta > c_{\beta+1} > \ldots > p$, 
	\item[]  $ I_\beta \subset [a_\beta, b_\beta]$ and  $ J_\beta \subset [d_\beta, c_\beta]$ for each $\beta < \gamma$.   \end{itemize}
Same as above,  choose numbers $F(a_n), F(b_n), F(d_n)$ and $F(c_n)$ belonging to $\mathbb R \setminus \textbf{X}_\gamma$, so that the part of the mapping $F$ defined in this way is increasing and  $F(q)=0$, and  $F(p)=1$. Then, using the inductive hypothesis,  extend $F$ to an increasing homeomorphism $F\colon Y \to \textbf{X}_\alpha$. \end{proof}

\begin{cor} \label{c24} Let  $Z\subset \mathbb R$ be a compact and   regularly closed set which satisfies condition $(\star)$.  If the derivative  $\mathcal I_Z^{(\alpha)}\not= \emptyset$ is finite, then $Z$ is homeomorphic to a union of finitely many copies of ${\mathbf X}_\alpha$. \end{cor}

\begin{proof} If $\mathcal I_Z^{(\alpha)} = \{g_1,g_2, \ldots, g_n\},$ then 
choose numbers $h_1, h_2, \ldots , h_{n-1}$  belonging to $ \mathbb R \setminus Z$ so that $$h_0< \inf Z< g_1 < h_1 < g_2 < \ldots < h_{n-1} < g_n < \sup Z< h_n.$$  By Proposition \ref{p24}, each set $Z\cap [h_i,h_{i+1}]$ is homeomorphic to  $\textbf{X}_\alpha$, therefore $Z$ is homeomorphic to a union of finitely many copies of ${\mathbf X}_\alpha$.
	\end{proof}

\begin{thm}\label{t25} Any $\mathbf{X}_\alpha$ is  a $t$-set.  
 \end{thm} 
\begin{proof} Suppose that $Y\subset \mathbb R$ is a copy of $\textbf{X}_\alpha$.
When $\alpha =0$, then $\textbf{X}_\alpha =[0,1]$ and $Y=[q,p] \subset \mathbb R$, so  any homeomorphism $F\colon \mathbb R \to \mathbb R$ such that $F(q)=0$ and $F(p)=1$ is good. 

When $\beta +1 = \alpha$, then $\mathcal I^{(\alpha)}_{\textbf{X}_\alpha} = \{[0,1]\}$ and let $\mathcal I^{(\alpha)}_{Y} = \{[q,p]\}$. Choose numbers $a_n, b_n, c_n$ and $d_n$  belonging to $\mathbb R\setminus Y$, as well as numbers $F(a_n), F(b_n), F(c_n)$ and $F(d_n)$  belonging to $\mathbb R\setminus \textbf{X}_\alpha$, same as in the proof of Proposition \ref{p24}, ensuring that the part of $F$ defined in this way is an increasing mapping.
  Then, using the inductive hypothesis,  extend $F$ to an increasing homeomorphism $F\colon \mathbb R \to \mathbb R$ such that
$$F(0)=q \mbox{ and }  F(1)=p, \mbox{ and } \{ F[I]\colon I \in \mathcal I^{(\beta)}_Y \} = \mathcal I^{(\beta)}_{\textbf{X}_\beta}.$$

Suppose that  $\alpha$ is a  limit ordinal,    $\mathcal I^{(\alpha)}_{Y} = \{[q,p]\}$,   $ I_\beta= \inf \mathcal I^{(\beta)}_Y$ and $ J_\beta= \sup \mathcal I^{(\beta)}_Y$.  For each $\beta < \alpha$, 
choose numbers $a_\beta, b_\beta, c_\beta$ and $d_\beta$  belonging to $\mathbb R\setminus Y$, as well as numbers $F(a_\beta), F(b_\beta), F(c_\beta)$ and $F(d_\beta)$  belonging to $\mathbb R\setminus \textbf{X}_\alpha$, same as in the proof of Proposition \ref{p24}, ensuring that the part of $F$ defined in this way is an increasing mapping.
  Then, using the inductive hypothesis,  extend $F$ to an increasing homeomorphism $F\colon \mathbb R \to \mathbb R$ such that
$F(0)=q$,  $F(1)=p$ and $\{ F[I]\colon I \in \mathcal I^{(\beta)}_Y \} = \mathcal I^{(\beta)}_{\textbf{X}_\beta}.$
  \end{proof}

  \begin{pro}\label{p26}
   There exists at least  $\omega_1$ many  $ t $-sets  regularly closed and compact subsets of $ \mathbb R$, which are not homeomorphic. \end{pro}
\begin{proof} The family $\{\textbf{X}_\alpha \colon \alpha < \omega_1\}$ consists of uncountably many regularly closed and compact subsets of $\mathbb R$ such that $\beta <\alpha$ implies that $\mathcal I^{(\beta)}_{\textbf{X}_\beta}$ is a singleton  and $\mathcal I^{(\beta)}_{\textbf{X}_\alpha}$ is infinite, therefore $\textbf{X}_\beta$ and $\textbf{X}_\alpha$ are not homeomorphic.   \end{proof}

\begin{thm}\label{t28} 
   There exists at most  $\omega_1$ many  $ t $-sets, which are not homeomorphic.
	\end{thm}
\begin{proof} Suppose $Y$ is a $ t $-set. If $\mathcal B_Y$ is not scattered, then $Y$ is an $M$-Cantorval, see Corollary \ref{c21}. Up to and including Proposition 19, all $t$-sets $Y$, with finite $\mathcal B_Y$, have been discussed and there are countably many of them up to homeomorphism. 
 It remains to consider the cases when $\mathcal B_Y$ is scattered and infinite, so let us now assume that $\mathcal B_Y$ is infinite and scattered. 
	
		If $Y$ is regularly closed, then  there exists $\alpha < \omega_1$ such that $Y$ is a union of finitely many pairwise disjoint closed-open sets containing exactly one interval from  $\mathcal B_{Y^{(\alpha)}}$. Hence, by Corollary \ref{c24}, $Y$ is a union of finitely many copies of $\mathbf X_\alpha$.  Thus, there are exactly  $\omega_1$-many non-homeomorphic   $t$-sets which are regularly closed.

	If there exists an isolated point in $Y$, then 	
	$Y^{(1)}$ is a regularly closed $ t $-set. Indeed,  arguing as in the proofs of Lemma \ref{le9} and Corollary \ref{c10}, we get that $\mathcal I_Y=\mathcal B_Y = \mathcal I_{Y^{(1)}}$. So, $Y \setminus Y^{(1)}$ is an open discrete subset of $Y$. Each  $Y^{(1)}$-gap  intersected  with $Y$ consists of   sequences, which converge to the endpoints of this gap. These properties are sufficient to check that $Y^{(1)}$ is a regularly closed $ t $-set. Thus, there are exactly  $\omega_1$-many non-homeomorphic   $t$-sets which are regularly closed
	
	Analogously, if there exists a closed-open subset of $Y$ which is a copy of the Cantor set,  then  let $Y^*$ be the union of all such closed-open subsets. Then 	
	$Y \setminus Y^{*}$ is a regularly closed $ t $-set. Indeed, by the definition, $Y \setminus Y^{*}$ is a  closed  subset of $Y$. There are no isolated points in $Y$ and $\mathcal C_Y=\mathcal D_Y =\emptyset$, hence  
	$\mathcal I_Y=\mathcal B_Y = \mathcal I_{Y\setminus Y^{*}}$ and $Y \setminus Y^{*}$ is a regularly closed  subset of $Y$. Each  $(Y\setminus Y^{*})$-gap  intersected  with $Y$ after enlarging by its endpoints is homeomorphic to the Cantor set. Just like before, these properties are sufficient  to check that $Y\setminus Y^{*}$ is a regularly closed $ t $-set. So, there are exactly  $\omega_1$-many non-homeomorphic   $t$-sets.	 
	\end{proof} 
	
	As it was noticed in \cite[Theorem 17]{bkp} there exist continuum many non-homeomorphic $L$-Cantorvals. Each $L$-Cantorval is  a compact subset of real numbers and the Continuum Hypothesis is undecidable within standard set theory. So,    Proposition \ref{p26} and Theorem \ref{t28} should be placed among  independence results
	in the sense that the existence of continuum  many $t$-sets is equivalent to the Continuum Hypothesis. 
	
As indicated at the beginning, the order of discussion gives a complete classification of $t$-sets up to homeomorphism. Namely, suppose $X$ is a non-empty $t$-set. By Proposition \ref{p11}, we have $\mathcal C_X = \emptyset$. 
\vspace{2mm}\\ \indent 
	 When $\mathcal Q_X$ is finite, then  $X$ consists of finitely many closed intervals, or a closed interval and an isolated point, or $X$ is finite.
	\vspace{1mm}\\  \indent When $\mathcal I_X = \emptyset$ and $\mathcal Q_X$ is infinite, then $X$ is a copy of the Cantor set. Indeed,  there are no isolated points and all components of $X$ are trivial, so  Proposition \ref{p12} works. 
\vspace{1mm} \\ \indent When $\mathcal B_X $ interlaces  $\mathcal Q_X$, then $X$ is an $M$-Cantorval, see Proposition \ref{p14} or compare Proposition \ref{p18}.
\vspace{1mm} \\ \indent When $\mathcal B_X$ is finite and non-empty, then $X$ is one of  the compact sets which are considered in Proposition \ref{p19}, i.e., $X$ is a finite union of copies  of either $S_1(\{0\})$ or $S_1([0,1])$, or $S_1(\mbox{The Cantor set})$. 
\vspace{1mm} \\ \indent When $\mathcal B_X$ is not scattered, then $X$ is an $M$-Cantorval, see Corollary \ref{c21}.
\vspace{1mm} \\ \indent
 When $\mathcal B_X$ is scattered and infinite, then   $X$ is a finite union of copies  of either $S_\alpha(\{0\})$ or $S_\alpha([0,1])$, or $S_\alpha(\mbox{The Cantor set})$, where $1<\alpha <\omega_1$. All the cases mentioned here follow from Corollary \ref{c24} and from the proof of Theorem \ref{t28}.

\textbf{Acknowledgment. } We would like to thank W. Bielas and T. Kania for their comments, which allowed us to remove typos and ambiguities from the original version published at arXiv:2602.15542.

\end{document}